\documentclass[11pt]{amsart}
\RequirePackage{amsmath}
\RequirePackage{amssymb}
\RequirePackage{amsthm}
\RequirePackage{epsfig}
\def\a{\alpha}               
\def\d{\delta}       \def\la{\lambda}     
       \def\t{\theta}       
                  
\def\ch{\chi}

\def\F{\Phi}

\def\D{{\mathbb D}}  
\def\C{{\mathbb C}}  
\def\R{{\mathbb R}}

\def\ch{{\mathcal H}}    
    
\def\rad{{\mathcal R}}   

\def\({\left(}       \def\){\right)}

\newtheorem{prop}{\sc Proposition}
\newtheorem{lem}[prop]{\sc Lemma}
\newtheorem{thm}[prop]{\sc Theorem}
\newtheorem{cor}[prop]{\sc Corollary}

\begin{document}
\title[On the generalized Zalcman functional]{On the generalized Zalcman functional for some classes of univalent functions}
\author[I. Efraimidis]{Iason Efraimidis}
\address{Departamento de Matem\'aticas, Universidad Aut\'onoma de
Madrid, 28049 Madrid, Spain}
\email{iason.efraimidis@uam.es}
\author[D. Vukoti\'c]{Dragan Vukoti\'c}
\address{Departamento de Matem\'aticas, Universidad Aut\'onoma de
Madrid, 28049 Madrid, Spain} \email{dragan.vukotic@uam.es}
 \urladdr{http://www.uam.es/dragan.vukotic}
\dedicatory{Dedicated to Professor David Shoikhet on the occasion of his 60th birthday}
\date{10 March, 2014.}
\begin{abstract}
We prove three sharp estimates for the generalized Zalcman coefficient functional: one for the Hurwitz class, another for the Noshiro-Warschawski class, and yet another for the functions in the closed convex hull of convex univalent functions. In each case the extremal functions are identified. We also observe that an asymptotic version of the Zalcman conjecture is true.
\end{abstract}
\maketitle
\par\medskip
\textbf{On the Zalcman functionals}. Let $\D$ denote the unit disk in the complex plane and $S$ the class of all normalized univalent (that is, analytic and one-to-one) functions in $\D$ with the Taylor series $f(z)=z+\sum_{n=2}^\infty a_n z^n$. See \cite{Du} and \cite{P} for the theory of these functions.
\par
It is well known that the coefficients of any $f$ in $S$ satisfy $|a_2^2-a_3| \le 1$; see \cite[Theorem~1.5]{P}. \textit{Zalcman's conjecture\/} states that every $f$ in $S$ should satisfy a more general inequality $|a_n^2-a_{2n-1}| \le (n-1)^2$. Partial progress was obtained in \cite{BT}, \cite{M1, M2}, and \cite{K1, K2}. The importance of the conjecture stems from the fact that, if assumed true, it implies the celebrated Bieberbach's conjecture:  $|a_n|\le n$ (see \cite{BT}). This problem was  instrumental in the development of the subject for several decades and was finally solved by L. de Branges in 1984 (\textit{cf.\/} \cite{dB} or \cite{H}). 
\par 
We observe in this paper that the Zalcman conjecture is asymptotically true.
\par
More general versions of Zalcman's conjecture have also been considered for the generalized functional $\F (f) = \la\,a_n^2 - a_{2n-1}$ for certain positive values of $\la$, \textit{e.g.\/}, in \cite{BT} and \cite{M2}. For this general functional, we obtain sharp estimates for three well-known classes of functions: for the Hurwitz and  Noshiro-Warschawski classes of univalent functions and for the closure of the convex hull of convex functions in $S$. These results, although simple and referring to special classes of functions, do not seem to have been recorded in the literature and it is our belief that they should be. Our exposition of results will flow from the more general to the more special classes and various methods of proof will be employed.
\par\medskip
\textsc{Acknowledgments}. {\small This paper is part of the first author's doctoral thesis work under the supervision of the second author at  Universidad Aut\'onoma de Madrid. The second author thankfully acknowledges partial support from MINECO grant MTM2012-37436-C02-02, Spain.}
\par\medskip
\textbf{An asymptotic version of the Zalcman conjecture}. Let $f\in S$ and denote by $M_\infty (r,f)$ the maximum modulus of $f$ on the circle of radius $r$ centered at the origin. Recall that the \textit{Hayman index\/} of $f$ is the number
$$
 \a=\lim_{r\to 1} (1-r)^2 M_\infty (r,f)\,.
$$
It is well known that $0\le\a\le 1$ \cite[p.~157]{Du}. Moreover, Hayman's regularity theorem \cite[Theorem~5.6]{Du} asserts that for each $f$ in $S$ its $n$-th Taylor coefficient $a_n$ satisfies $\lim_{n\to\infty} |a_n/n|=\a$.
\par
We now show that an asymptotic version of the Zalcman conjecture holds.
\begin{thm}
Let $f(z)=z+a_2 z^2+a_3 z^3+\ldots$, $f\in S$, and let $\a$ be its Hayman index. Then
\begin{equation}
 \lim_{n\rightarrow\infty} \frac{|a_n^2-a_{2n-1}|}{(n-1)^2} = \alpha^2\,.
 \label{haym}
\end{equation}
Also, if $B_n = \sup_{f\in S}|a_n^2-a_{2n-1}|$ then
$$
\lim_{n\rightarrow\infty} \frac{B_n}{(n-1)^2} = 1.
$$
\end{thm}
\begin{proof}
Applying the triangle inequality we get
\begin{align*}
\frac{|a_n^2-a_{2n-1}|}{(n-1)^2} & \leq \left(\frac{|a_n|}{n}\right)^2 \left( \frac{n}{n-1}\right)^2 + \frac{|a_{2n-1}|}{2n-1} \frac{2n-1}{(n-1)^2} \quad \text{and}\\
\frac{|a_n^2-a_{2n-1}|}{(n-1)^2} & \geq \left(\frac{|a_n|}{n}\right)^2 \left( \frac{n}{n-1}\right)^2 - \frac{|a_{2n-1}|}{2n-1} \frac{2n-1}{(n-1)^2},
\end{align*}
where the right-hand side in both inequalities converges to $\alpha^2$ in view of Hayman's regularity theorem. Hence \eqref{haym} follows.
\par
The Koebe function clearly shows that $B_n \geq (n-1)^2$. Using the customary notation $A_n=\sup_{f\in S} |a_n|$, we have
\begin{align*}
1 & \leq \frac{B_n}{(n-1)^2} \leq \frac{\sup_{f\in S}(|a_n|^2+|a_{2n-1}|)}{(n-1)^2} \leq \frac{A_n^2+A_{2n-1}}{(n-1)^2} \\
 & = \left(\frac{A_n}{n}\right)^2 \left( \frac{n}{n-1}\right)^2 + \frac{A_{2n-1}}{2n-1} \frac{2n-1}{(n-1)^2} \to 1\,, \quad n\to\infty\,.
\end{align*}
The convergence above holds, for example, in view of the asymptotic Bieberbach conjecture \cite[p.~66]{Du} if one does not want to use the full strength of de Branges' theorem.
\end{proof}
\begin{cor}
If $f\in S$ is not a rotation of the Koebe function, then for every $\d\in (0,1-\alpha^2)$ there exist $n_0\in\mathbb{N}$ such that
$$
 |a_n^2-a_{2n-1}| \leq (1-\delta)(n-1)^2\,,
$$
for all $n\geq n_0$.
\end{cor}
\par\medskip
\par\medskip
\textbf{Estimates for the closed convex hull of convex functions}. Denote by $C$ the class of convex functions in $S$. A typical example is the half-plane function $\ell(z)=\frac{z}{1-z}$. It is well known that the coefficient estimate can be improved a great deal for the functions in $C$: by a theorem of Loewner, they must satisfy $|a_n|\le 1$, with equality only for the function $\ell$ and its rotations (see \cite[Corollary on p.~45]{Du}. It would, thus, be natural to expect a similar improvement for convex functions in the Zalcman conjecture and this is indeed the case.
\par
Denote by $\textrm co(C)$ the convex hull of $C$ and by $\overline{\textrm co(C)}$ its closure in the topology of uniform convergence on compact subsets of $\D$. Note that this larger class  no longer consists only of univalent functions. 
\par 
By a \textit{rotation\/} of a function $f$ in $S$, as is usual \cite[Chapter~2]{Du}, we mean the function $f_c (z)=\overline{c} f(c z)$, $|c|=1$, which is again in $S$.
\par 
The estimate below is known when $\la=0$ so there is no need to include that case in the result.
\par\medskip
\begin{thm} \label{thm-co-convex}
Let $0<\la\le 2$. If $f \in \overline{\textrm co(C)}$, then $|\la\,a_n^2 - a_{2n-1}| \le 1$ for all $n\ge 2$. For any fixed $n$ and $\la <2$, equality holds only for the functions of the following form (and for their rotations): 
$$
 f(z) = \sum_{k=1}^{2n-2} m_k \frac{z}{1- e^{i \theta_k}z},
$$
where $0\leq m_k \leq 1, \, \theta_k=\frac{(2k+1) \pi}{2n-2}$, and
$$
 \sum_{k=1}^{n-1} m_{2k} = \sum_{k=1}^{n-1} m_{2k-1} = 1/2\,.
$$
\end{thm}
\begin{proof}
A direct proof can be given by following the methods from \cite{BT}. By a well-known result from \cite{BMW}, there is a unique probability measure $\mu$ on $[0,2\pi]$, such that
$$
 f(z) = \int_0^{2\pi} \frac{z}{1- e^{i \theta} z} d\mu(\theta)
$$
for all $f$ in $\overline{\mathrm{co}(C)}$. It easily follows by a geometric series expansion that the coefficients of $f$ are
$$
 a_n = \int_0^{2\pi} e^{i(n-1)\theta} d\mu(\theta), \quad n\geq 2.
$$
Since the functional $\F (f) = \la\,a_n^2 - a_{2n-1}$ has the property that $\F (f_c)=c^{2n-2} \F(f)$, after an appropriate rotation we can consider instead the problem of maximizing the quantity
$$
 M_n = \textrm{Re}\,\{\la\,a_n^2 - a_{2n-1}\}\,.
$$
Applying the Cauchy-Schwarz inequality to the first integral below, we obtain
\begin{eqnarray*}
 M_n &=& \la\,\left(\int_0^{2\pi} \cos\bigl((n-1)\theta\bigr) d\mu(\theta) \right)^2 -  \la\,\left(\int_0^{2\pi} \sin\bigl((n-1)\theta\bigr) d\mu(\theta) \right)^2
\\
 & & - \int_0^{2\pi} \cos\bigl(2(n-1)\theta\bigr) d\mu(\theta)
\\
 &\le & \la\,\int_0^{2\pi} \cos^2\bigl((n-1)\theta\bigr) d\mu(\theta) - 2 \int_0^{2\pi} \cos^2\bigl((n-1)\theta\bigr) d\mu(\theta) + 1
\\
 & = & (\la - 2) \int_0^{2\pi} \cos^2\bigl((n-1)\theta\bigr) d\mu(\theta) + 1
\\
 &\le & 1
\end{eqnarray*}
since $\la\le 2$ by assumption.
\par
We now check that the estimate is sharp. When $\la<2$, in order for equality to hold in the second inequality, $\cos\bigl((n-1)\theta\bigr) = 0$ must hold on the support of the measure $\mu$. That is, $\mu$ must be supported on the set
$$
 \left\{ \frac{(2k+1) \pi}{2n-2} : k=1,2,\ldots,2n-2 \right\}\,,
$$
possibly on a proper subset of it. Denote by $m_k$, $k=1$, $2$,\ldots,$2n-2$ the point masses corresponding to each of the points above. Then  $m_k\ge 0$ for all $k$ in $\{1,2,\ldots,2n-2\}$ and $\sum_{k=1}^{2n-2} m_k=1$ since $\mu$ is a probability measure and every candidate for an extremal function must be of the form
$$
 f(z) = \sum_{k=1}^{2n-2} m_k \frac{z}{1- e^{i \theta_k}z}\,.
$$
In order to have equality in the first inequality, the following must also hold:
$$
 0 = \int_0^{2\pi} \sin\bigl((n-1)\theta\bigr) d\mu(\theta) = \sum_{k=1}^{2n-2} m_k \sin\bigl((k+\frac12)\,\pi\bigr) = \sum_{k=1}^{2n-2} (-1)^k m_k \,.
$$
If the last two conditions are fulfilled, it is readily seen that equality indeed holds throughout.
\end{proof}

\par\medskip
\textbf{Estimates for the Noshiro-Warschawski class}. We now consider the functions in the normalized class
$$
 \rad = \{ f \in \ch(\D)\,\colon\,\mathrm{Re} f^\prime (z) > 0, f(0) = f^\prime (0)-1 = 0\}\,.
$$
A typical example of a function in $\rad$ is $f(z)=2\log\dfrac1{1-z} - z$ whose derivative is $f^\prime(z)= (1+z)/(1-z)$, a mapping of $\D$ onto the right half-plane. The branch of the logarithm is defined so that $\log 1=0$.
\par 
Note that $\rad\subset S$ by the basic Noshiro-Warschawski lemma \cite[Theorem~2.16]{Du}. It is  a ``small'' subclass of $S$ so the coefficients of the functions in it are ``small''. See, for example,  MacGregor's paper \cite{MG}, where it was shown that for $f$ in $\rad$ we have $|a_n|\le 2/n$. Thus, it is also natural to expect a smaller optimum estimate on the size of the Zalcman functional in $\rad$ than the one conjectured for $S$. This is indeed the case, as our next result shows.
\par\medskip
\begin{thm} \label{thm-nw}
If $0<\la\le 4/3$ and $f\in \rad$, then for all $n\ge 2$ we have
$$
 |\la\,a_n^2 - a_{2n-1}| \leq \frac{2}{2n-1}\,.
$$
For $\la<4/3$ and any fixed $n\ge 2$, equality holds only for the functions of the following form (and for their rotations): 
$$
 f(z) = \sum_{k=1}^{2n-2} m_k \left(2 e^{-i \theta_k} \log\frac{1}{1- e^{i \theta_k}z} -z\right) = \sum_{k=1}^{2n-2} 2 m_k e^{-i \theta_k} \log\frac{1}{1- e^{i \theta_k}z} -z\,,
$$
where
$$
 0\leq m_k \leq 1\,, \qquad \t_k=\frac{(2k+1)\,\pi}{2n-2}\,, \qquad \sum_{k=1}^{n-1} m_{2k} = \sum_{k=1}^{n-1} m_{2k-1} = 1/2\,.
$$
\end{thm}
\begin{proof} 
By the Herglotz representation theorem for functions with positive real part \cite[\S~1.9]{Du}, there is a unique probability measure $\mu$ on $[0,2\pi]$, such that
$$
 f^\prime(z) = \int_0^{2\pi} \frac{e^{i t} + z}{e^{i t} - z} d\mu(t) = \int_0^{2\pi} \frac{1+ e^{i \theta} z}{1- e^{i \theta} z} d\mu(\theta)
$$
or, equivalently,
$$
1+\sum_{n=2}^\infty n a_n z^{n-1} = 1+ \sum_{n=1}^\infty 2 \int_0^{2\pi}  e^{i n\theta} d\mu(\theta) z^n.
$$
This shows that the coefficients of $f$ are
$$
a_n = \frac{2}{n} \int_0^{2\pi} e^{i(n-1)\theta} d\mu(\theta), \quad n\geq 2.
$$
We can again consider maximizing the real-valued functional
$$
 M_n = \textrm{Re}\,\{\la\,a_n^2 - a_{2n-1}\}\,.
$$
After applying the Cauchy-Schwarz inequality, we get
\begin{align*}
 M_n &= \frac{4\,\la}{n^2} \left(\int_0^{2\pi} \cos\bigl((n-1)\theta\bigr) d\mu(\theta) \right)^2 - \frac{4\,\la}{n^2}  \left(\int_0^{2\pi} \sin\bigl((n-1)\theta\bigr) d\mu(\theta) \right)^2
\\
 &- \frac{2}{2n-1} \int_0^{2\pi} \cos\bigl(2(n-1)\theta\bigr) d\mu(\theta)
\\
 & \leq \frac{4\,\la}{n^2} \int_0^{2\pi} \cos^2\bigl((n-1)\theta\bigr) d\mu(\theta) -  \frac{2}{2n-1} \left( 2 \int_0^{2\pi} \cos^2\bigl((n-1)\theta\bigr) d\mu(\theta) -1  \right)
\\
 & = \frac{4((2n-1)\,\la - n^2)}{n^2(2n-1)}   \int_0^{2\pi} \cos^2\bigl((n-1)\theta\bigr) d\mu(\theta) + \frac{2}{2n-1} \leq  \frac{2}{2n-1}\,,
\end{align*}
since, in view of our  assumption that $\la\le 4/3$, we have
$(2n-1)\la - n^2 \le 0$, with strict inequality whenever $\la <4/3$.
\par
In the case $\la <4/3$, an inspection of the above chain of inequalities reveals that in order to have equality everywhere we must have $\cos\bigl((n-1)\theta\bigr) = 0$ on the support of $\mu$, that is
$$
 supp(\mu) \subset \left\{ \frac{(2k+1)\,\pi}{2n-2} : k=1,2,\ldots,2n-2 \right\}\,,
$$
and also
$$
 0 = \int_0^{2\pi} \sin\bigl((n-1)\theta\bigr) d\mu(\theta) = \sum_{k=1}^{2n-2} m_k \sin\bigl((k+\frac12)\,\pi\bigr) = \sum_{k=1}^{2n-2} m_k (-1)^k.
$$
Hence,
$$
 f(z) = \sum_{k=1}^{2n-2} m_k \int_{[0,z]}  \frac{1+ e^{i \theta_k} \zeta}{1- e^{i \theta_k} \zeta} d\zeta = \sum_{k=1}^{2n-2} m_k \left( 2 e^{-i \theta_k} \log\frac{1}{1- e^{i \theta_k}z} -z \right)\,.
$$
\end{proof}

\textbf{Estimates for the Hurwitz class}. The name \textit{Hurwitz class\/} is used by various authors to denote the set $\ch$ of all functions $f$ of the form
$$
 f(z)=z+a_2 z^2+ a_3 z^3+\ldots\,,
$$
analytic in $\D$ and with the property that $\sum_{n=2}^\infty n |a_n|\le 1$. Obviously, the $n$-th coefficient of a function in $\ch$ is subject to the estimate $|a_n|\le 1/n$ for each $n$. The simplest example of a function in $\ch$ is the polynomial $P_n (z)= z + \frac{z^n}{n}$, $n\ge 2$.
\par 
It is a well-known exercise that $\ch\subset S$. Actually, more is true: $\ch\subset\rad$. This can be seen as follows. If $f$ is a function in $\ch$ other than the identity, then  $f^\prime(0)=1$ and, when $z\neq 0$, we have the strict inequality 
$$ 
 \textrm{Re}\,f^\prime (z) = 1 + \sum_{n=2}^\infty n \textrm{Re}\,\{a_n z^{n-1}\} \ge 1 - \sum_{n=2}^\infty n |a_n| |z|^{n-1} > 1 - \sum_{n=2}^\infty n |a_n| \ge 0\,. 
$$
The reader is referred to \cite{G} for further properties of $\ch$. 
\par 
We now consider the generalized Zalcman functional for the functions in this class. One may again expect a better estimate for the functions in $\ch$ than for those in $\rad$. In order to prove such an inequality, the following simple lemma from the calculus of two variables will be useful.
\par\medskip
\begin{lem} \label{l-calc}
Let $\lambda>0$, $n\ge 2$, and consider the triangle
$$
 \Delta=\{(u,v)\in\R^2\,\colon\,u\ge 0,\, v\ge 0,\, nu+(2n-1)v \le 1\}
$$
in the $uv$-plane. Then
$$
 \max_{(u,v)\in\Delta} (\lambda u^2+v) = \max\left\{ \frac{\la}{n^2},
 \frac{1}{2n-1}\right\}\,,
$$
and equality can hold only at the points $(u,v)=(0,\frac{1}{2n-1})$ and $(u,v)=(\frac{1}{n},0)$.
\end{lem}
\begin{proof}
The function $F(u,v)=\lambda u^2+v$ is readily seen to have no critical points so its maximum on the compact set $\Delta$ is achieved on the boundary $\partial\Delta$.
\par
Clearly, $F(0,v)=v\le \frac1{2n-1}$ while $F(u,0)=\lambda u^2 \le \frac{\lambda}{n^2}$.
\par
Finally, on the third piece of the boundary of $\Delta$ we have $nu+(2n-1)v=1$, hence the function $F$ there can be seen as a function of one variable
$$
 F(u,v)=g(u)=\lambda u^2+\frac{1-nu}{2n-1}\,.
$$
Since $g^{\prime\prime}(u)=2\lambda>0$, the above function cannot achieve its maximum even if its critical point $\frac{n}{2 (2n-1) \lambda}$ is an interior point of the admissible interval $[0,\frac1{2n-1}]$ for $u$. Hence, the maximum value can only be achieved at one of the endpoints of this interval. Since
$$
 g(0)=\frac{1}{2n-1}\,, \quad g\(\frac{1}{n}\)=\frac{\lambda}{n^2}\,,
$$
and $u=1/n$ if and only if $v=0$, the statement follows.
\end{proof}
\par\medskip
\begin{thm} \label{thm-h}
If $\la>0$ and $f\in \ch$, then for each $n \ge 2$ we have
\begin{equation}
 |\la\,a_n^2 - a_{2n-1}| \leq \max\left\{ \frac{\la}{n^2},
 \frac{1}{2n-1}\right\}\,.
 \label{ineq-max}
\end{equation}
Equality holds if and only if
$$
f(z) =
\begin{cases}
z + \frac{\alpha}{2n-1} \, z^{2n-1}, \quad \text{for} \; \;
\la\leq\frac{n^2}{2n-1}  \\[\jot] z + \frac{\alpha}{n} \, z^n,
\quad\qquad \text{for} \; \; \la \geq \frac{n^2}{2n-1},
\end{cases}
$$
where $\alpha$ is a complex number such that $|\alpha|=1$. 
\par 
Note that the rotations are already included in the form of extremal functions. Also, two different types of extremal functions exist in the case $\la=\frac{n^2}{2n-1}$.
\end{thm}
\begin{proof}
Consider the open set
$\Omega = \{(z,w)\in\C^2\,\colon\,n|z| + (2n-1) |w|<1\}$. By standard metric/topological arguments, it is easily checked that its boundary is the set $\partial\Omega = \{(z,w)\in\C^2\,\colon\,n|z| + (2n-1) |w|=1\}$.
\par
By the defining condition of the Hurwitz class, the ordered pair of coefficients $(a_n, a_{2n-1})$ of an arbitrary function in $\ch$ must belong to the set
$$
 \overline{\Omega} = \{(z,w)\in\C^2\,\colon\,n|z| + (2n-1) |w| \leq 1\} \,.
$$
The function $\F(z,w)=\la z^2 - w$ is holomorphic in $\C^2$, hence by  the Maximum Modulus Principle it follows that $\max_{\overline{\Omega}} |\F| = \max_{\partial \Omega} |\F|$.
\par
Note that, upon applying the rotation which transforms $f$ into $f_c$, $|c|=1$, the ordered pair $(a_n,a_{2n-1})$ gets converted into $(A_n,A_{2n-1})=(c^{n-1} a_n,c^{2n-2} a_{2n-1})$. It is clear that $(a_n,a_{2n-1})\in \overline{\Omega}$ if and only if $(A_n,A_{2n-1})\in \overline{\Omega}$. Thus, in view of homogeneity of the generalized Zalcman functional: $\F (f_c)=c^{2n-2} \F(f)$, instead of maximizing $|\la\,a_n^2 - a_{2n-1}|$, we may consider maximizing the quantity
$$
 |\F (a_{n},a_{2n-1})| = \textrm{Re}\, \F (A_n,A_{2n-1}) = \la (x^2-y^2) - r\,,
$$
where $A_n = c^{n-1} a_n = x+yi$ and $A_{2n-1} = c^{2n-2} a_{2n-1} = r+si$. By Lemma~\ref{l-calc} it follows that
\begin{equation}
 |\F (a_{n},a_{2n-1})| \le \la (x^2+y^2) + |r| \le \la |A_n|^2 + |A_{2n-1}| \le \max\left\{ \frac{\la}{n^2},\frac{1}{2n-1}\right\}\,.
 \label{ineq-final}
\end{equation}
In view of these observations, we deduce that \eqref{ineq-max} holds.
\par
In order for equality to hold in the final estimate \eqref{ineq-max}, equality must hold in all intermediate inequalities. Thus, in particular, we must have $(a_n,a_{2n-1})\in \partial \Omega$; that is, $n|a_n| + (2n-1) |a_{2n-1}|=1$. Together with the defining condition of $\ch$, it follows that the rotated function must be of the form $f_c (z) = z + A_n z^n + A_{2n-1} z^{2n-1}$. Further inspection of the case of equality in Lemma~\ref{l-calc} and the values of $\la$ readily yields that one of the coefficients $A_n$, $A_{2n-1}$ must be zero and the more precise form of these functions follows immediately.
\end{proof}


\end{document}